\documentclass[12pt]{amsart}
\usepackage{amssymb,latexsym}
\usepackage{amsbsy,amssymb,amscd,amsfonts,latexsym,amstext,delarray, amsmath,color,caption,enumerate,amsaddr,subcaption,wrapfig}
\usepackage{amsthm}
\usepackage{mathrsfs}
\usepackage[matrix,arrow,curve]{xy}

\usepackage{graphicx}
\usepackage{xcolor}
\usepackage{tcolorbox}

\numberwithin{equation}{section} 

\newtheorem{theorem}{Theorem}[section]
\newtheorem{proposition}[theorem]{Proposition}

\newtheorem{corollary}[theorem]{Corollary}
\newtheorem{lemma}[theorem]{Lemma}

\theoremstyle{definition}

\newtheorem{remark}[theorem]{Remark}

\newtheorem{definition}[theorem]{Definition}

\newtheorem{problem}[theorem]{Problem}

\newcommand{\Dj}{\hbox to 8pt{\raisebox{.4\height}{-}\hss D}}

\newcommand\beq{\begin{equation}}
\newcommand\eeq{\end{equation}}
\newcommand\bqa {\begin{eqnarray}}
\newcommand\eqa {\end{eqnarray}}

\newcommand{\bear}{\begin{array}}
\newcommand{\enar}{\end{array}}

\usepackage{graphicx}
\usepackage{wrapfig}

\begin{document}

\title{Noncommutative cross-ratio and Schwarz derivative}

\author{Vladimir Retakh, Vladimir Rubtsov, Georgy Sharygin}
\date{}

\maketitle

\begin{center}
{\footnotesize\textit{Dedicated to Emma Previato.}}
\end{center}

\begin{abstract}
We present here a theory of noncommutative cross-ratio, Schwarz derivative and their connections and relations to the operator cross-ratio. We apply the theory to ``noncommutative elementary geometry" and relate it to noncommutative integrable systems. We also provide a noncommutative version of the celebrated ``pentagramma mirificum''.

\end{abstract}



\section{Introduction}

Cross-ratio and Schwarz derivative are some of the most famous
invariants in mathematics (see \cite{Lab}, \cite{OT}, \cite{OT2}).
Different versions of their noncommutative analogs and various
applications of these constructions to integrable systems, control theory and other subjects
were discussed in several publications including \cite{DGP2}. In this
paper, which is the first one of a series of works, we recall some of these definitions, revisit the previous
results and discuss their connections with each other and with
noncommutative elementary geometry. In the forthcoming papers we shall further discuss the role of noncommutative cross ratio in the theory of noncommutative integrable models and in topology.

The present paper is organized as follows. In Sections 2, 3 we recall a
definition of noncommutative cross-ratios based on the theory of
noncommutative quasi-Pl\"ucker invariants (see \cite{GR2, GGRW}), in Section 4
we use the theory of quasideterminants (see  \cite{GR1}) to obtain
noncommutative versions  of Menelaus's and Ceva's theorems. In Section
5 we compare our definition of cross-ratio with the operator
version used in control theory \cite{Zelikin06} and show how Schwarz derivatives appear as the infinitesimal analogs of noncommutative cross-ratios. In section 6 we revisit an approach to noncommutative Schwarz derivative from \cite{RS} and section 7 deals with possible applications of the theory we develop. It should also be mentioned that in present paper we develop the constructions and ideas, first outlined in \cite{RS}.

It is our pleasure to dedicate this paper to Emma Previato, whose intelligence, erudition, interest to various domaines of our science are spectacular and her friendship is constant and
loyal. Her results (\cite{DGP2}) were one of important motives which  inspired us to think once more about the role of non-commutative cross-ratio.

\vskip 2mm \noindent{\bf Acknowledgements.} The authors are grateful to B. Khesin, V. Ovsienko and S. Tabachnikov  for helpful discussions.
This research was started during V. Retakh's visit to LAREMA and Department of Mathematics, University of Angers. He is thankful to the project DEFIMATH for its support and LAREMA
for hospitality. V.~Roubtsov thanks the project IPaDEGAN (H2020-MSCA-RISE-2017), Grant Number 778010 for support of his visits to CRM, University of Montreal where the paper was finished and the CRM group of Mathematical Physics  for hospitality. He is partly supported  by the Russian Foundation for Basic Research under the Grants RFBR 18-01-00461.
G.~Sharygin is thankful to IHES and LAREMA for hospitality during his visits. His research is partly supported  by the Russian Science Foundation, Grant No. 16-11-10069.

\section{Quasi-Pl\"ucker coordinates}

We begin with a list of basic properties of noncommutative cross-ratios introduced in \cite{R}. To this end we first recall the definition and properties of quasi-Pl\"ucker coordinates; observe that we shall only deal with the quasi-Pl\"ucker coordinates for $2\times n$-matrices over a noncommutative division ring $\mathcal R$. The corresponding theory for general $k\times n$-matrices is presented in \cite {GR2, GGRW}. 

Recall (see \cite {GR, GR1} and subsequent papers) that for a matrix $\begin{pmatrix} a_{1k}&a_{1i}\\ a_{2k}&a_{2i}\end{pmatrix}$ one can define four quasideterminants provided the corresponding elements are invertible:
$$
\begin{vmatrix} \boxed{a_{1k}}&a_{1i}\\ a_{2k}&a_{2i}\end{vmatrix}=a_{1k}-a_{1i}a_{2i}^{-1}a_{2k}, \ \
\begin{vmatrix} a_{1k}&\boxed{a_{1i}}\\ a_{2k}&a_{2i}\end{vmatrix}=a_{1i}-a_{1k}a_{2k}^{-1}a_{2i},
$$
$$
\begin{vmatrix} a_{1k}&a_{1i}\\ \boxed{a_{2k}}&a_{2i}\end{vmatrix}=a_{2k}-a_{2i}a_{1i}^{-1}a_{1k},\ \ 
\begin{vmatrix} a_{1k}&a_{1i}\\ a_{2k}&\boxed{a_{2i}}\end{vmatrix}=a_{2i}-a_{2k}a_{1k}^{-1}a_{1i}.
$$

Let $A=\begin{pmatrix} a_{11}&a_{12}&\dots&a_{1n}\\a_{21}&a_{22}&\dots&a_{2n}\end{pmatrix}$ be a matrix over $\mathcal R$.
\begin{lemma} Let $i\neq k$. Then 
$$
\begin{vmatrix} a_{1k}&\boxed{a_{1i}}\\ a_{2k}&a_{2i}\end{vmatrix}^{-1}   
\begin{vmatrix} a_{1k}&\boxed{a_{1j}}\\ a_{2k}&a_{2j}\end{vmatrix}=
\begin{vmatrix} a_{1k}&a_{1i}\\ a_{2k}&\boxed{a_{2i}}\end{vmatrix}^{-1}
\begin{vmatrix} a_{1k}&a_{1j}\\ a_{2k}&\boxed{a_{2j}}\end{vmatrix}^{-1}
$$
if the corresponding expressions are defined.     
\end{lemma}

Note that in the formula the boxed elements on the left and on the right must be in the same row. 

\begin{definition} We call the expression
$$
q_{ij}^k(A)=\begin{vmatrix} a_{1k}&\boxed{a_{1i}}\\ a_{2k}&a_{2i}\end{vmatrix}^{-1}   
\begin{vmatrix} a_{1k}&\boxed{a_{1j}}\\ a_{2k}&a_{2j}\end{vmatrix}=
\begin{vmatrix} a_{1k}&a_{1i}\\ a_{2k}&\boxed{a_{2i}}\end{vmatrix}^{-1}
\begin{vmatrix} a_{1k}&a_{1j}\\ a_{2k}&\boxed{a_{2j}}\end{vmatrix}^{-1}
$$
the quasi-Pl\"ucker coordinates of matrix $A$. 
\end{definition}

Our terminology is justified by the following observation. Recall that
in the commutative case the expressions 
$$
p_{ik}(A)=\begin{vmatrix}a_{1i}&a_{1k}\\ a_{2i}&a_{2k}\end{vmatrix}=
a_{1i}a_{2k}-a_{1k}a_{2i}
$$
are the Pl\"ucker coordinates of $A$. One can see that in the commutative case
$$
q_{ij}^k(A)=\frac{p_{jk}(A)}{p_{ik}(A)},
$$ 
i.e. quasi-Pl\"ucker coordinates are ratios of Pl\"ucker coordinates.

Let us list here the properties of quasi-Pl\"ucker coordinates over (noncommutative)
division ring $\mathcal R$. For the sake of brevity we shall sometimes write $q_{ij}^k$ instead of $q_{ij}^k(A)$ where it cannot lead to a confusion.

\begin{enumerate}
\item Let $g$ be an invertible matrix over $\mathcal R$. Then
$$ q_{ij}^k(g\cdot A)=q_{ij}^k(A).$$

\item Let $\Lambda = \text {diag}\ (\lambda_1, \lambda _2,\dots, \lambda_n)$ be an
invertible diagonal matrix over $\mathcal R$. Then
$$ q_{ij}^k(A\cdot \Lambda)=\lambda_i^{-1}\cdot q_{ij}^k(A)\cdot \lambda_j.$$

\item If $j=k$ then $q_{ij}^k=0$; if $j=i$ then $q_{ij}^k=1$ (we always assume
$i\neq k$).

\item $q_{ij}^k\cdot q_{j\ell}^k=q_{i\ell}^k\ $. In particular, $q_{ij}^kq_{ji}^k=1$.

\item ``Noncommutative skew-symmetry": For distinct $i,j,k$
$$
q_{ij}^k\cdot q_{jk}^i\cdot q_{ki}^j=-1.
$$
One can also rewrite this formula as $q_{ij}^kq_{jk}^i=-q_{ik}^j$.

\item ``Noncommutative Pl\"ucker identity": For distinct $i,j,k,\ell$
$$
q_{ij}^k q_{ji}^{\ell} + q_{i\ell}^k q_{\ell i}^j=1.
$$
\end{enumerate}

\medskip\noindent
One can easily check two last formulas in the commutative case. In fact,
$$
q_{ij}^k\cdot q_{jk}^i\cdot q_{ki}^j=
\frac{p_{jk}p_{ki}p_{ij}}{p_{ik}p_{ji}p{kj}}=-1
$$
because Pl\"ucker coordinates are skew-symmetric: $p_{ij}=-p_{ji}$ for
any $i,j$.

Also, assuming that $i<j<k<\ell$
$$
q_{ij}^k q_{ji}^{\ell} + q_{i\ell}^k q_{\ell i}^j=
\frac{p_{jk}p_{i\ell}}{p_{ik}p_{j\ell}}+
\frac{p_{\ell k}p_{ij}}{p_{ik}p_{\ell j}}.
$$
Since $\frac{p_{\ell k}}{p_{\ell j}}=\frac {p_{k\ell}}{p_{j\ell}}$,
the last expression is equal to
$$
\frac{p_{jk}p_{i\ell}}{p_{ik}p_{j\ell}}+
\frac{p_{k\ell }p_{ij}}{p_{ik}p{j\ell }}=
\frac {p_{ij}p_{k\ell}+p_{i\ell}p_{jk}}{p_{ik}p_{j\ell}}=1
$$
due to the celebrated Pl\"ucker identity
$$
p_{ij}p_{k\ell} - p_{ik}p_{j\ell} +p_{i\ell}p_{jk}=0.
$$

\begin{remark} We present here the theory of the {\it left}
quasi-Pl\"ucker coordinates for $2$ by $n$ matrices where $n>2$.
The theory of the {\it right} quasi-Pl\"ucker coordinates for $n$ by $2$
or, more generally, for $n$ by $k$ matrices where $n>k$ can be found
in \cite{GR2, GGRW}.
\end{remark}

\section{Definition and basic properties of cross-ratios}

\subsection{Non-commutative cross-ratio: basic definition.} 

We define cross-ratios over (noncommutative) division ring $\mathcal R$
by imitating the definition of classical cross-ratios in homogeneous coordinates.
Namely, if four points in (real or complex) projective plane can be represented in homogeneous coordinates by vectors $a,b,c,d$ such that $c=a+b$ and $d=ka+b$, then their cross-ratio is $k$.

So we let
$$
x=\begin{pmatrix} x_1\\ x_2\\ \end{pmatrix},\ \
y=\begin{pmatrix} y_1\\ y_2\\ \end{pmatrix},\ \
z=\begin{pmatrix} z_1\\ z_2\\ \end{pmatrix},\ \
t=\begin{pmatrix} t_1\\ t_2 \end{pmatrix}
$$
be four vectors in $\mathcal R^2$. We define their cross-ratio $\kappa=\kappa (x,y,z,t)$
by equations
$$
\begin{cases}
t=x\alpha +y\beta\\
z=x\alpha \gamma +y\beta \gamma \cdot \kappa
\end{cases}
$$
where $\alpha, \beta, \gamma, \kappa \in \mathcal R$.

In order to obtain explicit formulas, let us consider the matrix
$$
\begin{pmatrix} 
x_1&y_1&z_1&t_1\\
x_2&y_2&z_2&t_2
\end{pmatrix}.
$$
We shall identify its columns with $x,y,z,t$. Then we have the following theorem (see \cite{R})
\begin{theorem}
$$
\kappa (x,y,z,t)=q_{zt}^y\cdot q_{tz}^x\ .
$$
\end{theorem}
Note that in the generic case 
$$
\begin{aligned}
\kappa &(x,y,z,t)= \begin{vmatrix} y_1&\boxed{z_1}\\
y_2&z_2\end{vmatrix}^{-1}
\begin{vmatrix} y_1&\boxed{t_1}\\
y_2&t_2\end{vmatrix}\cdot
\begin{vmatrix} x_1&\boxed{t_1}\\
x_2&t_2\end{vmatrix}^{-1}
\begin{vmatrix} x_1&\boxed{z_1}\\
x_2&z_2\end{vmatrix}\\
&=z_2^{-1}(z_1z_2^{-1}-y_1y_2^{-1})^{-1}(t_1t_2^{-1}-y_1y_2^{-1})(t_1t_2^{-1}-x_1x_2^{-1})^{-1}(z_1z_2^{-1}-x_1x_2^{-1})z_2
\end{aligned}
$$
which shows that $\kappa(x,y,z,t)$ coincides with the standard cross-ratio in commutative case and also demonstrates the importance of conjugation in the noncommutative world. 

\begin{corollary} Let $x,y,z,t$ be vectors in $\mathcal R$,
$g$ be a $2$ by $2$ matrix over $\mathcal R$ and
$\lambda_i\in \mathcal R$, $i=1,2,3,4$. If the matrix $g$
and elements $\lambda_i$ are invertible then
\begin{equation}
\kappa (gx\lambda_1, gy\lambda_2, gz\lambda_3, gt\lambda_4)=
\lambda_3^{-1}\kappa (x,y,z,t)\lambda_3\ .
\end{equation}
\end{corollary}
Again, as expected,  in the commutative case the right hand side of (3.1) equals
$\kappa (x,y,z,t)$.


\begin {remark}
Note that the group $GL_2(\mathcal R)$ acts on vectors in $\mathcal R^2$
by multiplication from the left: $(g,x)\mapsto gx$,  and the group $\mathcal R^{\times}$ of
invertible elements in $\mathcal R$ acts by multiplication from the right:
$(\lambda, x)\mapsto x\lambda^{-1}$. These actions determine the action of
$GL_2(\mathcal R)\times T_4(\mathcal R)$ on 
$P_4=\mathcal R^2\times \mathcal R^2\times \mathcal R^2\times \mathcal R^2$ where
$T_4(\mathcal R)=(\mathcal R^{\times})^4$.
The cross-ratios are {\it relative invariants} of the action.
\end{remark}

The following theorem generalizes the main property of cross-ratios to the noncommutive case (see \cite{R}).
\begin{theorem} Let $\kappa (x,y,z,t)$ be defined and $\kappa (x,y,z,t)\neq 0,1$.
Then $4$-tuples $(x,y,z,t)$ and $(x',y',z',t')$ from
$P_4$ belong to the same orbit of $GL_2(\mathcal R)\times T_4(\mathcal R)$ if and only if
there exists $\mu \in \mathcal R^{\times}$ such that
\begin{equation}
\kappa (x,y,z,t)=\mu \cdot \kappa (x',y',z',t')\cdot \mu ^{-1}\ .
\end{equation}
\end{theorem}

The following corollary shows that the cross-ratios we defined satisfy {\it cocycle conditions}
(see \cite{Lab}).
\begin{corollary} For all vectors $x,y,z,t,w$ the following equations hold
\begin{align*}
\kappa(x,y,z,t)=\kappa (w,y,z,t)\kappa(x,w,z,t)\\
\kappa (x,y,z,t)=1-\kappa(t,y,z,x),
\end{align*}
if all the cross-ratios in these formulas exist.
\end{corollary}

The last proposition can also be generalized as follows:
\begin{corollary} For all vectors $x,x_1, x_2,\dots x_n, z,t\in \mathcal R^2$
one has $$\kappa (x,x,z,t)=1$$ and
$$
\kappa(x_{n-1},x_n,z,t)\kappa(x_{n-2}, x_{n-1},z,t)\dots \kappa(x_1,x_2,z,t)=
\kappa(x_1,x_n,z,t)
$$
where we assume that all the cross-ratios exist.
\end{corollary}

\subsection{Noncommutative cross-ratios and permutations}

There are $24$ cross-ratios defined for vectors $x,y,z,t\in \mathcal R^2$, if we permute them.
They are related by the following formulas:

\begin{proposition} Let $x,y,z,t\in \mathcal R$. Then
\begin{align}
q_{tz}^x\kappa (x,y,z,t)q_{zt}^x=
q_{tz}^y\kappa (x,y,z,t)q_{zt}^y=\kappa (y,x,t,z);\\
q_{xz}^y\kappa (x,y,z,t)q_{zx}^y=
q_{xz}^t\kappa (x,y,z,t)q_{zx}^t=\kappa (z,t,x,y);\\
q_{yz}^x\kappa (x,y,z,t)q_{zy}^x=
q_{yz}^t\kappa (x,y,z,t)q_{zy}^t=\kappa (t,z,x,y);\\
\kappa (x,y,z,t)^{-1}=\kappa (y,x,z,t).
\end{align}
\end{proposition} 
Note again the appearance of conjugation in the noncommutative case; this happens since
$q_{ij}^k$ and $q_{ji}^k$ are inverse to each other. Also observe that using Proposition 3.7 and the cocycle condition (corollary 3.5) one can get all 24 formulas for cross-ratios of $x,y,z,t$ knowing just one of them.

\subsection{Noncommutative triple ratio}

Let $\mathcal R$ be a division ring as above; we shall work with the  ``right $\mathcal R$-plane'' $\mathcal R^2$, i.e. we use right multiplication of vectors by the elements from $\mathcal R$. Consider the triangle with vertices $O(0,0),\ X(x,0)$, and $Y(0,y)$ in $\mathcal R^2$. Let $A(a_1,a_2)$ be a point on side $XY$, $B(b,0)$ be a point on side $OX$ and $C(0,c)$ be a point on side $OY$. Recall that the geometric condition $A\in XY$ means
$$x^{-1}a_1+y^{-1}a_2=1.$$

Let $P(p_1,p_2)$ be the point of intersection of $XC$ and $YB$. Then one has
$$
p_1=(y-c)(yb^{-1}-cx^{-1})^{-1}, \ p_2=(x-b)(by^{-1}-xc^{-1})^{-1}\ .
$$

Let $Q$ be the point of intersection of $OP$ and $XY$. The non-commutative cross ratio for $Y,A,Q,X$ is equal to
$$
x^{-1}(1-p_1p_2^{-1}a_2a_1^{-1})^{-1}x.
$$
By changing the order of $Y,A,Q,X$ we get up to a conjugation
\begin{equation}
p_1p_2^{-1}a_2a_1^{-1}=-(y-c)(yb^{-1}-cx^{-1})^{-1}(x-b)^{-1}xc^{-1}(yb^{-1}-cx^{-1})bx^{-1}(x-a_1)a_1^{-1} .
\end{equation}
In the commutative case (up to a sign) we have
$$
p_1p_2^{-1}a_2a_1^{-1}=(y-c)c^{-1}b(x-b)^{-1}(x-a_1)a_1^{-1}
$$
(compare it with the Ceva theorem in elementary geometry).

Note that $(x-a_1)^{-1}a_1^{-1}=YA/AX$ and (3.7) is a (non-commutative analogue of) triple cross-ration (see section 6.5 in the book by Ovsienko and Tabachnikov \cite{OT}) 
 
\subsection {Noncommutative angles and cross-ratios.}

Let $\mathcal R$ be a noncommutative division ring. Recall that noncommutative angles (or noncommutative $\lambda$-lengths) $T_i^{jk}=T_i^{kj}$ for vectors $A_1,A_2, A_3,A_4\in {\mathcal R}^2,\ A_i=(a_{1i},a_{2i})$ are defined by the formulas
$$
T_i^{jk}=x_{ji}^{-1}x_{jk}x_{ik}^{-1}.
$$
Here $x_{ij}= a_{1j} - a_{1i}a_{2i}^{-1}a_{2j}$, or $x_{ij} = a_{2j} - a_{2i}a_{1i}^{-1}a_{1j}$ (see \cite{BR18}). On the other hand the cross-ratio $\kappa (A_1,A_2,A_3,A_4)=\kappa(1,2,3,4)$ (see definition 2.2) is
$$
\kappa(1,2,3,4)=q_{34}^2q_{43}^1.
$$
It implies that
$$
\kappa (1,2,3,4)=x_{43}^{-1}(T_4^{23})^{-1}T_4^{31}x_{43}\ .
$$
In other words, cross-ratio is a ratio of two angles up to a conjugation. 

Under the transformation $x_{ij}\mapsto \lambda _ix_{ij}$ we have
$$
T_i^{jk}\mapsto T_i^{jk}\cdot \lambda_i^{-1}\ .
$$
Also note that
$$
T_i^{jk}(T_i^{mk})^{-1}=q_{ik}^jq_{ki}^m
$$
i.e. $T_i^{jk}(T_i^{mk})^{-1}$ is a cross-ratio. Further details on the properties of $T_i^{jk}$ can be found in \cite{BR18}.

\section{ Noncommutative Menelaus'  and Ceva's theorems}
\bigskip

\subsection{Higher rank quasi-determinants: reminder}
Let $A=(a_{ij})$, $i,j=1,2,\dots,n$ be a matrix over a ring. Denote by $A^{pq}$ the submatrix of matrix $A$ obtained from $A$ by removing the
$p$-th row and the $q$-th column. Let $r_p=(a_{p1}, a_{p2},\dots, \hat a_{pq},\dots a_{pn})$ be the row submatrix and 
$c_q=(a_{1q}, a_{2q},\dots, \hat a_{pq},\dots a_{nq})^T$ be the column submatrix of $A$.  Following \cite{GR} we say that the quasideterminant
$|A|_{pq}$ is defined if and only if the submatrix $A^{pq}$ is invertible. In this case
$$
|A|_{pq}=a_{pq}-r_p(A^{pq})^{-1}c_q\ .
$$

In the commutative case $|A|_{pq}=(-1)^{p+q}det A/det A^{pq}$. It is sometimes convenient to use the notation
$$
|A|_{pq}=\left | \begin{matrix} \dots &\dots&a_{1q}&\dots\\
\dots&\dots&\dots&\dots\\
a_{p1}&\dots&\boxed {a_{pq}}&\dots\\
\dots&\dots&\dots &\dots \end{matrix}\right |\ .
$$

\medskip
\subsection{Commutative Menelaus'  and Ceva's theorems}
We follow the affine geometry proof. Let the points $D,E,F$ lie on the straight lines $AB,BC$ and $AC$ respectively (see figure 1 (\textit{a})). Denote by  $\lambda_D$ the coefficient for homothety with center $D$ sending $B$ to $C$, by $\lambda_E$ the coefficient for homothety with center $E$ sending $C$ to $A$, and by $\lambda_F$ the coefficient for homothety with center $F$ sending $A$ to $B$. Note that (in a generic case)
\[
\begin{aligned}
\lambda_D&=(b_i-d_i)^{-1}(c_i-d_i), \  i=1,2,\\
\lambda_E&=(c_i-e_i)^{-1}(a_i-e_i), \  i=1,2,\\
\lambda_F&=(a_i-f_i)^{-1}(b_i-f_i), \  i=1,2.
\end{aligned}
\]
Here $(a_1,a_2)$ are the coordinates of $A$ etc. We shall omit the indices and write $\lambda_D=(b-d)^{-1}(c-d)$, etc.

\medskip
\begin{theorem}\label{ComMenelaus}
 Points $E,D,F$ belong to a straight line if an only if 
$$
(a-f)^{-1}(b-f)\cdot (c-e)^{-1}(a-e)\cdot (b-d)^{-1}(c-d)=1\ .
$$
\end{theorem}

\begin{figure}[tb]
\includegraphics[scale =0.8]{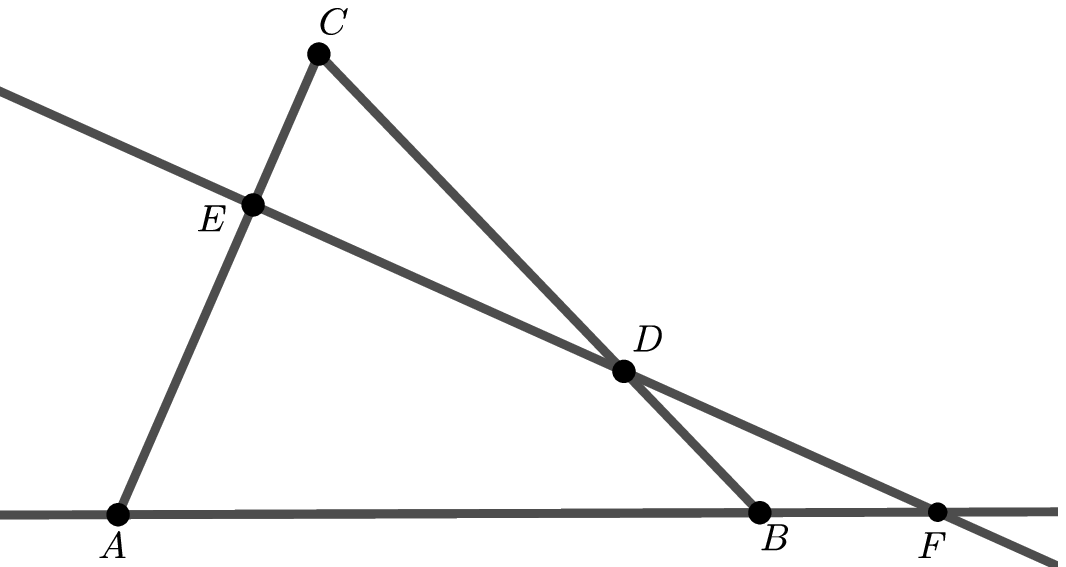}\ \includegraphics[scale=0.8]{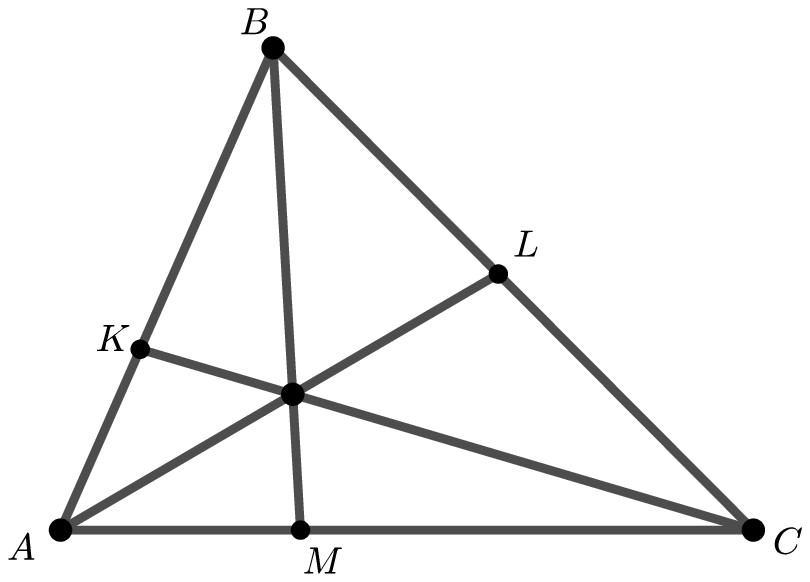}\\
\begin{center}
(\textit{a})\hspace{8cm}(\textit{b})
\end{center}
\caption{The classical Menelaus (part (\textit{a})) and Ceva (part (\textit{b})) theorems.}
\end{figure}
\noindent This is the Menelaus' theorem in the commutative case.
\begin {proof}
The composition of transformations $\lambda_D, \lambda_E, \lambda _F$ leaves the point $B$ unchanged, thus it is equal to a homothety with center $B$. On the other hand, if points belong to the same straight line then the center should belong to this line too, so the composition is equal to identity. So
$$\lambda_F \lambda_E \lambda _D=1$$
i.e.
$$
(a-f)^{-1}(b-f)\cdot (c-e)^{-1}(a-e)\cdot (b-d)^{-1}(c-d)=1.
$$
The opposite statement can be proved by contradiction.
\end{proof}
Somewhat dually, one obtains Ceva theorem (see figure 1 (\textit{b})):
\begin{theorem}\label{ComCeva}
 Lines $AD$, $BE$ and $CF$ intersect each other in a point $O$ if and only if
$$
(e-a)^{-1}(e-c)\cdot (f-b)^{-1}(f-a)\cdot (d-c)^{-1}(d-b)= -1\ .
$$
\end{theorem}
This is the Ceva' theorem in the commutative case.

\subsection{Non-commutative Menelaus'  and Ceva's theorems}





Let $\mathcal R$ be a noncommutative division ring. Consider $\mathcal R^2$ as the right vector space over $\mathcal R$. For a point $X\in \mathcal R^2$ denote by $x_i$ its $i$-th coordinate, $i=1,2$. Here and below we shall use the properties of quasideterminants, see \cite{GR,GR1}:
\begin{proposition}\label{NCalign}
 Let points $X$ and $Y$ are in generic position, i.e. that matrix 
$$\left (\begin{matrix}x_1&y_1\\
x_2&y_2\end{matrix} \right )
$$
is invertible. Then
the points $X,Y,Z\in\mathcal R^2$ belong to the same straight line (in the sense of linear algebra) if and only if
$$
\left | \begin{matrix} x_1 & y_1 & z_1\\
x_2 &y_2 & z_2\\
1 & 1 & {\boxed 1} \end{matrix}\right | =0\ .
$$
\end{proposition}
\begin{proof} From the general theory of quasideterminants it follows that that
$$
\left | \begin{matrix} x_1 & y_1 & z_1\\
x_2 &y_2 & z_2\\
1 & 1 & {\boxed 1} \end{matrix}\right | = 1 - \lambda -\mu
$$
where $\lambda, \mu\in R$ satisfy the equation $X\lambda + Y\mu =Z$.
Note that $X$, $Y$ and $Z$ belong to the same straight line if and only if there exists $\lambda + \mu =1$.
\end{proof} 

\begin{corollary}\label{NCaligncond}
Assume that $x_i-y_i\in\mathcal R,\ i=1,2$ are invertible. Then $X,Y,Z$ belong to one straight line if and only if
$$
(y_1-x_1)^{-1}(z_1-x_1)=(y_2-x_2)^{-1}(z_2-x_2)\ .
$$
\end{corollary}
\begin{proof} Note that
$$
\left | \begin{matrix} x_1 & y_1 & z_1\\
x_2 &y_2 & z_2\\
1 & 1 & {\boxed 1} \end{matrix}\right | = 
-\left |\begin{matrix} x_1&y_1-x_1\\
\boxed {x_2}&y_2-x_2\end{matrix} \right |^{-1}\cdot \left 
|\begin{matrix} y_1-x_1&z_1-x_1\\
y_2-x_2&\boxed {z_2-x_2}\end{matrix} \right |
$$
and that $(y_1-x_1)^{-1}(z_1-x_1)=(y_2-x_2)^{-1}(z_2-x_2)$ if and only if
$$
\left |\begin{matrix} y_1-x_1&z_1-x_1\\
y_2-x_2&\boxed {z_2-x_2}\end{matrix} \right | = 0\ .
$$
\end{proof} 

\subsection{NC analogue of Konopelchenko equations}
Let again $\mathcal R$ be a division ring. Consider $\mathcal R^2$ as the right module over $\mathcal R$.

\begin{proposition}
Let $F_1=(x_1,y_1),\ F_2=(x_2,y_2)$ be two points in $\mathcal R^2$ in a generic position. Then the equation of the straight line $L_{12}$ passing through $F_1$ and $F_2$ is
$$
(y_2-y_1)^{-1}(y-y_1)=(x_2-x_1)^{-1}(x-x_1).
$$
\end{proposition}
\begin{corollary}. An equation of the line $L_{12}'$ parallel to $L$ and passing through $(0,0)$ is
$$
(y_2-y_1)^{-1}y=(x_2-x_1)^{-1}x,
$$
i.e. any point $F_{12}$ on $L_{12}'$ has coordinates $((x_2-x_1)f_{12}, (y_2-y_1)f_{12})$.
\end{corollary}

The proposition and the corollary are both straightforward consequences of the Proposition \ref{NCalign} and the Corollary \ref{NCaligncond}.

Denote by $L_{ij}$ the straight line passing through $F_i=(x_i,y_i)$ and $F_j=(x_j,y_j)$ and by
$L_{ij}'$ the parallel line though $(0,0)$. Consider now (additionaly to the line $L_{12}'$ and to a point $(x_2-x_1)f_{12}, (y_2-y_1)f_{12}$ on it) points $F_{23}=((x_3-x_2)f_{23}, (y_3-y_2)f_{23})$ on line $L_{23}'$ and $F_{31}=(x_1-x_3)f_{31}, (y_1-y_3)f_{31})$ on line $L_{31}'$.

\begin{proposition} 
For generic points $F_1,F_2,F_3$ the points $F_{12}, F_{23}, F_{31}$ belong to a straight line iff
$$
f_{12}^{-1} + f_{23}^{-1} + f_{31}^{-1}\ = 0 .
$$
\end{proposition}
\medskip

\noindent{\bf Warning:} Note that before we considered points with coordinates $(x_1,x_2)$, $(y_1,y_2)$, $(z_1,z_2)$ and now
with coordinates $(x_i, y_i)$.
\begin{proof} 
According to  Proposition \ref{NCalign} in order to show that $F_{12}$, $F_{23}$, $F_{31}$ lie on the same straight line it is necessary and sufficient to check that
$$
\theta: =\left |\begin{matrix} (x_2-x_1)f_{12}&(x_3-x_2)f_{23}&(x_1-x_3)f_{31}\\
(y_2-y_1)f_{12}&(y_3-y_2)f_{23}&(y_1-x_3)f_{31}\\
1&1&\boxed {1}\end{matrix} \right | = 0.
$$
According to the standard properties of quasideterminants
$$
\theta=
\left |\begin{matrix} (x_2-x_1)&(x_3-x_2)&(x_1-x_3)\\
(y_2-y_1)&(y_3-y_2)&(y_1-y_3)\\
f_{12}^{-1}&f_{23}^{-1}&\boxed {f_{31}^{-1}}\end{matrix} \right |f_{31}.
$$
Adding the first two columns to the third one does not change $\theta$, so
$$
\theta=
\left |\begin{matrix} (x_2-x_1)&(x_3-x_2)&0\\
(y_2-y_1)&(y_3-y_2)&0\\
f_{12}^{-1}1&f_{23}^{-1}&\boxed {f_{12}^{-1}+f_{21}^{-1}+f_{31}^{-1}}\end{matrix} \right |f_{31}=
(f_{12}^{-1}+f_{21}^{-1}+f_{31}^{-1})f_{31}.
$$
\end{proof} 
This is a noncommutative generalization of formula (32) from Konopelchenko (\cite{Konop})

\subsection{Noncommutative Menelaus theorem and quasi-Pl\"ucker coordinates}
Let as above $\mathcal R$ be a noncommutative division ring. Consider $\mathcal R^2$ as the right vector space over $\mathcal R$.
For a point $X\in\mathcal  R^2$ denote by $x_i$ its $i$-th coordinate, $i=1,2$.
Recall that points $X,Y,Z$ are collinear if and only if
$$
\left | \begin{matrix} x_1 & y_1 & z_1\\
x_2 &y_2 & z_2\\
1 & 1 & {\boxed 1} \end{matrix}\right | =0\ .
$$

\begin{corollary} 
\label{cor1}
Points $X,Y,Z$ are collinear if and only if
$$
(y_1-x_1)^{-1}(z_1-x_1)=(y_2-x_2)^{-1}(z_2-y_2)
$$
or, equivalently,
$$
(x_2-y_2)^{-1}(z_2-y_2)=q_{XZ}^Y.
$$
\end{corollary}

\noindent
\begin{remark}The second identity is equivalent to the equality
$$
\left | \begin{matrix} x_1 & y_1 & {\boxed {z_1}}\\
x_2 &y_2 & z_2\\
1 & 1 & 1\end{matrix}\right | =0\ .
$$ 
\end{remark}

\begin{proposition}\label{prop1}
 Let $A,B,C$ be non-collinear points in $\mathcal R^2$. Then any point $P\in\mathcal R^2$ can be uniquely written as 
$$P=At+Bu+Cv, \ \ t,u,v\in \mathcal R, \ \ t+u+v=1\ .$$
\end{proposition}
We will write $P=[t,u,v]$.
\begin{proposition}\label{kap}
 Let $P_i=[t_i,u_i,v_i]$, $i=1,2,3$. Then $P_1,P_2,P_3$ are collinear if and only if
$$
\left | \begin{matrix} t_1 & t_2 & {\boxed {t_3}}\\
u_1 &u_2 & u_3\\
v_1 & v_2 & v_3\end{matrix}\right | =0\ .
$$ 
\end{proposition}
We follow now the book by Kaplansky, \cite{Kap}, see pages 88-89. Consider a triangle $ABC$ (vertices go anti-clock wise). Take point $R$ at line $AB$, point $P$ at line
$BC$, and point $Q$ at line $AC$. Then
$$
P=B(1-t)+Ct, \ Q=C(1-u)+Au, \ R=A(1-v)+Bv\ .
$$
Proposition \ref{kap} implies

\begin{theorem}\label{thmABC}
Points $A,B,C$ are collinear if and only if
$$
u(1-u)^{-1}t(1-t)^{-1}v(1-v)^{-1}=-1\ .
$$
\end{theorem}
Note that $t(1-t)^{-1}=(c_1-p_1)^{-1}(p_1-b_1)$. Corollary \ref{cor1} implies that 
$$t(1-t)^{-1}=-q_{CB}^P$$
where $q_{CB}^P$ is a quasi-Pl\"ucker coordinate. Similarly,
$$
u(1-u)^{-1}=-q_{AC}^Q, \ \ v(1-v)^{-1}=-q_{BA}^R
$$
and Theorem \ref{thmABC} implies
\begin{theorem}
$$
q_{AC}^Qq_{CB}^Pq_{BA}^R=1\ .
$$
\end{theorem}
  
\section{Relation with matrix cross-ratio.}

In this section we discuss the cross ratio in noncommutative algebras, introduced above in terms of \textit{quasideterminants} and its relation with the \textit{operator cross-ratio} of Zelikin (see \cite{Zelikin06} and also Chapter 5 of \cite{Zelikin-book}). We also consider the infinitesimal part of the cross-ratio in this case, which just like in the classical case will lead us to the construction of the noncommutative Schwarz derivative. Further details about the constructions of Schwarzian will be given in the next section.

Recall that 
we defined the cross ratio of four elements $a,b,c,d\in \mathcal R^{\oplus 2}$ by explicit formulas as follows:
\[
\kappa(a,b,c,d)=\begin{vmatrix}b_1 & \boxed{d_1}\\ b_2 & d_2\end{vmatrix}^{-1}\begin{vmatrix}b_1 & \boxed{c_1}\\ b_2 & c_2\end{vmatrix}\begin{vmatrix}a_1 & \boxed{c_1}\\ a_2 & c_2\end{vmatrix}^{-1}\begin{vmatrix}a_1 & \boxed{d_1}\\ a_2 & d_2\end{vmatrix},
\]
under the assumption that all these expressions exist (in fact, except for the existence of the inverse elements of $a_2$ and $b_2$, it is enough to assume further that the matrices $\begin{pmatrix}a_1 & c_1\\ a_2 & c_2\end{pmatrix}$ and $\begin{pmatrix}b_1 & d_1\\ b_2 & d_2\end{pmatrix}$ are invertible).

The expression $\kappa(a,b,c,d)$ has various algebraic properties (see sections 1-3).
\ We are going now to compare it with the \textit{operator cross ratio} of Zelikin (see \cite{Zelikin06}). To this end we begin with the description of his construction.

Let $\mathcal H$ be an even-dimensional (possibly infinite-dimensional) vector space; let us fix its polarization $\mathcal H=V_0\oplus V_1$, where the subspaces $V_0,\,V_1$ have the same dimension (in infinite dimensional case one can assume that there is a fixed isomorphism $\psi:V_0\to V_1$ between them); let $(\mathscr P_1,\mathscr P_2)$ and $(\mathscr Q_1,\mathscr Q_2)$ be two other pairs of subspaces, polarizing $\mathcal H$, i.e. $\mathscr P_i,\,\mathscr Q_i$ are isomorphic to $V_j$ and $\mathscr P_1$ (resp. $\mathscr Q_1$) is transversal to $\mathscr P_2$ (resp. to $\mathscr Q_2$). Then the cross ratio of these two pairs (or of the spaces $\mathscr P_1\,\mathscr P_2,\,\mathscr Q_1,\,\mathscr Q_2$ is the operator
\[
\mathrm{DV}(\mathscr P_1,\mathscr P_2,\mathscr Q_1,\mathscr Q_2)=(\mathscr P_1\stackrel{\mathscr P_2}{\to}\mathscr Q_1\stackrel{\mathscr Q_2}{\to}\mathscr P_1)
\]
Here we use the notation from [9], where $\mathscr P_1\stackrel{\mathscr P_2}{\to}\mathscr Q_1$ denotes the projection of $\mathscr P_1$ to $\mathscr Q_1$ along $\mathscr P_2$ and similarly for the second arrow. 

In the cited paper the following explicit formula for $\mathrm{DV}$ was proved: let $\mathscr P_i$ be given by the graph of an operator $P_i:V_0\to V_1,\ i=1,2$ and similarly for $\mathscr Q_j$, then the following formula holds:
\[
\mathrm{DV}(\mathscr P_1,\mathscr P_2,\mathscr Q_1,\mathscr Q_2)=(P_1-P_2)^{-1}(P_2-Q_1)(Q_1-Q_2)^{-1}(Q_2-P_1):V_0\to V_0.
\]
The invertibility of the operators $P_1-P_2$ and $Q_1-Q_2$ is provided by the transversality of $\mathscr P_1$ and $\mathscr P_2$ (resp. $\mathscr Q_1$ and $\mathscr Q_2$).

The first claim we are going to make is the following:
\begin{proposition}
The operator cross ratio $\mathrm{DV}(\mathscr Q_2,\mathscr P_2,\mathscr Q_1,\mathscr P_1)$ (if it exists) is equal to $\kappa(p_1,p_2,q_1,q_2)$, for $p_1=\begin{pmatrix} 1\\ P'_1\end{pmatrix},\ p_2=\begin{pmatrix} 1\\ P'_2\end{pmatrix},\ q_1=\begin{pmatrix} 1\\ Q'_1\end{pmatrix},\ q_2=\begin{pmatrix} 1\\ Q'_2\end{pmatrix}$, where $1$ is the identity operator on $V_0$ and we identify $V_0$ and $V_1$ using the fixed map $\psi$ so that $P'_i=\psi^{-1}\circ P_i:V_0\to V_0$.
\end{proposition}
\begin{proof}
This is a direct computation based on the explicit formula:
\[
\begin{aligned}
\kappa(p_1,p_2,q_1,q_2)&=\begin{vmatrix}1 & \boxed{1}\\ P'_2 & Q'_2\end{vmatrix}^{-1}\begin{vmatrix}1 & \boxed{1}\\ P'_2 & Q'_1\end{vmatrix}\begin{vmatrix}1 & \boxed{1}\\ P'_1 & Q'_1\end{vmatrix}^{-1}\begin{vmatrix} 1 & \boxed{1}\\ P'_1 & Q'_2\end{vmatrix}\\
                                         &=(1-(P_2')^{-1}Q'_2)^{-1}(1-(P_2')^{-1}Q_1')(1-(P_1')^{-1}Q_1')^{-1}(1-(P_1')^{-1}Q_2')\\
                                         &=(1-P_2^{-1}Q_2)^{-1}(1-P_2^{-1}Q_1)(1-P_1^{-1}Q_1)^{-1}(1-P_1^{-1}Q_2)\\
                                         &=(P_2-Q_2)^{-1}P_2P_2^{-1}(P_2-Q_1)(P_1-Q_1)^{-1}P_1P_1^{-1}(P_1-Q_2)\\
                                         &=(Q_2-P_2)^{-1}(P_2-Q_1)(Q_1-P_1)^{-1}(P_1-Q_2)\\
                                         &=\mathrm{DV}(\mathscr Q_2,\mathscr P_2,\mathscr Q_1,\mathscr P_1).
\end{aligned}
\]
\end{proof}
Observe, that the role of $\psi$ is insignificant here: in effect, one can define the quasideterminants in the context of categories, i.e. for $A$ being a matrix of morphisms in certain category with its entries $a_{ij}$ being maps from the $i$-th object to the $j$-th object (see \cite{GGRW}). This makes the use of $\psi$ redundant.

\subsection{Cocycle identity: cross-ratio and ``classifying map''}
It is shown in \cite{Zelikin06} that the following equality holds for the $\mathrm{DV}$: let $(\mathscr P_1,\mathscr P_2)$ be a polarizing pair, and $\mathscr X,\,\mathscr Y,\,\mathscr Z$ three hyperplanes, then
\begin{equation}
\label{eq:coc1}
\mathrm{DV}(\mathscr P_1,\mathscr X,\mathscr P_2,\mathscr Y)\,\mathrm{DV}(\mathscr P_1,\mathscr Y,\mathscr P_2,\mathscr Z)\,\mathrm{DV}(\mathscr P_1,\mathscr Z,\mathscr P_2,\mathscr X)=1,
\end{equation}
or, using the algebraic properties of $\mathrm{DV}$,
\begin{equation}
\label{eq:coc2}
\mathrm{DV}(\mathscr P_1,\mathscr X,\mathscr P_2,\mathscr Y)\,\mathrm{DV}(\mathscr P_1,\mathscr Y,\mathscr P_2,\mathscr Z)=\mathrm{DV}(\mathscr P_1,\mathscr X,\mathscr P_2,\mathscr Z)
\end{equation}
if all three terms are well-defined. This relation can be reinterpreted topologically, by saying that the operator cross ratio corresponds to the change of coordinates function in the tautological fibre bundle over the Grassmanian space of polarizations of $\mathcal H$ (see \cite{Zelikin06}). 

Recall now that the noncommuative cross-ratio $\kappa(a,b,c,d)$ verifies a similar relation
\begin{equation}
\label{eq:coc3}
\kappa(y,x,p_2,p_1)\kappa(z,y,p_2,p_1)=\kappa(z,x,p_2,p_1),
\end{equation}
see Sections 2 and 3 above for a purely algebraic proof. One can ask, if there exists an analogous topological interpretation of $\kappa$, i.e. if one can construct an analog of tautological bundle in the purely algebraic case. Here we shall sketch a construction, intended to answer this question, postponing the details to a forthcoming paper, dealing with the topological applications of the noncommutative cross-ratio. 

In order to give an interpretation of relations \eqref{eq:coc1}-\eqref{eq:coc3}, let us fix a vector $\omega=\begin{pmatrix}\omega_1\\ \omega_2\end{pmatrix}\in \mathcal R^{\oplus 2}$; let $\mathcal R^2_\omega$ denote the set of the elements $\begin{pmatrix}a_1\\ a_2\end{pmatrix}\in \mathcal R^{\oplus 2}$ such that the matrix
\[
\begin{pmatrix}\omega_1 & a_1\\ \omega_2 & a_2\end{pmatrix}
\]
is invertible. Let $a\in \mathcal R^2_\omega$ and let $x\in\mathcal R^{\oplus 2}$ be such that both matrices
\[
\begin{pmatrix}a_1 & x_1\\ a_2 & x_2\end{pmatrix}\ \mbox{and}\ \begin{pmatrix}\omega_1 & x_1\\ \omega_2 & x_2\end{pmatrix}
\]
are invertible. It is clear that the set of such $x$ is equal to the intersection $\mathcal R^2_\omega\bigcap \mathcal R^2_a$; we shall denote it by $\mathcal R^2_{a,\omega}=\tilde{\mathcal R}^2_a$ since $\omega$ is fixed.

Consider now a \v Cech type simplicial complex $\check C_\cdot(\mathcal R^2)$: its set of $n$-simplices is spanned by the disjoint union of the intersections 
\[
\check C_n(\mathcal R^2;\omega)=\coprod_{a_0,\dots,a_n}\tilde {\mathcal R}^2_{a_0}\bigcap \tilde {\mathcal R}^2_{a_1}\bigcap\dots\bigcap\tilde {\mathcal R}^2_{a_n},
\]
and the faces/degenracies are given by the omitting/repeating the terms in the intersections respectively.

Then the formula
\[
\phi=\{\phi_{a_0,a_1}\}:\check C_1(\mathcal R^2;\omega)\to\mathcal R^*,\ \phi_{a_0,a_1}(x)=\kappa(a_1,a_0,x,\omega),
\]
determines a map on the second term of this complex. Observe, that the cocycle condition now can be interpreted as the statement that $\phi$ can be extended to a simplicial map from $\check C_\cdot(\mathcal R^2;\omega)$ to the bar-resolution of the group $\mathcal R^*$ of invertible elements in $\mathcal R$. Namely: put
\begin{align*}
\phi_0=1&:\check C_0(\mathcal R^2;\omega)\to [1]=B_0(\mathcal R^*);\\
\phi_1=\phi&:\check C_1(\mathcal R^2;\omega)\to\mathcal R^*=B_1(\mathcal R^*);\\
\intertext{and for all other $n\ge 2$}
\phi_n&:\check C_n(\mathcal R^2;\omega)\to(\mathcal R^*)^{\times n}=B_n(\mathcal R^*)\\
\intertext{given by the formula}
\phi_n(x)&=[\phi_{a_0,a_1}(x)|\phi_{a_1,a_2}(x)|\dots|\phi_{a_{n-1},a_n}(x)],
\end{align*}
for all $x\in \tilde{\mathcal R}^2_{a_0,\dots,a_n}$. Then
\begin{proposition}
The collection of maps $\{\phi_n\}_{n\ge 0}$ determine a simplicial map from $\check C_\cdot(\mathcal R^2,\omega)$ to $B_\cdot(\mathcal R^*)$.
\end{proposition}
\begin{remark}
The construction we just described bears striking similarity with the well-known Goncharov's complex (see \cite{Gonch}), so one can wonder if there are any relation with the actual Goncharov's Grassmannian complex and higher cross ratios/polylogarithms in this case?
\end{remark}


\subsection{Schwarzian operator}
In classical theory Schwarzian operator is a quadratic differential operator, measuring the ``non-projectivity'' of a diffeomorphism of (real or complex) projective line. Applying the same ideas to noncommutative plane, we shall obtain an analog of operator as an infinitesimal part of the deformation of the cross-ratio. From the construction we use it then follows that this operator is invariant with respect to the action of $GL_2(\mathcal R)$ and the multiplication by invertible elements from $\mathcal R$.

First, following the ideas in \cite{Zelikin06} we consider a smooth one-parameter family $Z(t)=\begin{pmatrix}Z(t)_1\\ Z(t)_2\end{pmatrix}$ of elements in $\mathcal R^{\oplus 2}$, such that for all different $t_1,t_2,t_3,t_4$ the cross ratio $\kappa(Z(t_1),Z(t_2),Z(t_3),Z(t_4))$ is well defined. Then, let us consider the function
\[
\begin{aligned}
f(t,t_1,t_2,t_3)&=\kappa(Z(t_3),Z(t_1),Z(t),Z(t_2))\\
                       &=(z(t_2) -z(t_1))^{-1}(z(t_1) -z(t))(z(t) -z(t_3))^{-1}(z(t_3) - z(t_2)),
\end{aligned}
\]
where $z(t)=Z(t)_1^{-1}Z(t)_2$. Fix $t=0$, and let $t_2\to 0$. Then $f(0,t_1,t_2,t_3)\to 1$ and 
\[
\frac{\partial f}{\partial t_2}(0,t_1,0,t_3)=-(z(0)-z(t_1))^{-1}z'(0)+(z(0)-z(t_3))^{-1}z'(0).
\]
Thus, 
\[
f(t,t_1,t_2,t_3)=1-(t_2-t)\left((z(t)-z(t_1))^{-1}z'(t)-(z(t)-z(t_3))^{-1}z'(t)\right)+o(t_2-t)
\]
If $t_1=t_3$, the derivative on the right vanishes; consider the second partial derivative:
\[
\frac{\partial^2f}{\partial t_3\partial t_2}(0,t_1,0,t_1)=-(z(0)-z(t_1))^{-1}z'(t_1)(z(0)-z(t_1))^{-1}z'(0),
\]
so that
\[
f(t,t_1,t_2,t_3)=1-(t_2-t)(t_3-t_1)(z(t)-z(t_1))^{-1}z'(t_1)(z(t)-z(t_1))^{-1}z'(t)+o((t_2-t)(t_1-t_3)).
\]
This expression has a singularity at $t_1=0$. Now, using the Taylor series for $z(t)$ we compute for $t_1\to 0$:
\[
\frac{\partial^2f}{\partial t_3\partial t_2}(0,t_1,0,t_1)=t_1^{-2}\left(1+\frac{t_1^2(z'(0))^{-1}z'''(0)}{6}-\frac{t_1^2((z'(0))^{-1}z''(0))^2}{4}+...\right)
\]
where $...$ denote the terms of degrees $3$ and higher in $t_1$. So, we obtain
\begin{equation}
\label{eq:schwa1}
\frac{\partial^2f}{\partial t_3\partial t_2}(0,t_1,0,t_1)=t_1^{-2}(1+6t_1^2S(Z)+...)+...
\end{equation}
where we put
\[
S(Z)=(z'(0))^{-1}z'''(0)-\frac32((z'(0))^{-1}z''(0))^2.
\]
Here $Z$ and $z$ are related as explained above. This differential operator is well-defined on functions with values in $\mathcal R^{\oplus 2}$, it is invariant with respect to the action of $GL_2(\mathcal R)$ and is conjugated by $\lambda\in\mathcal R^\times$, when $Z$ is multiplied by it on the right.

Thus we come up with the following statement:
\begin{proposition}
Suppose we have a $1$-parameter family of elements in the projective noncommutative plane $\mathcal R^{\oplus 2}$, then the infinitesimal part of the cross ratio of four generic points in this family is equal to the noncommutative Schwarzian $S(Z)$.
\end{proposition}
\begin{proof}
Above we have given a sketch, explaining the formula for $S(Z)$; however, the way we obtained the  formula \eqref{eq:schwa1} there was a bit artificial. Now let us consider the formal Taylor expansion of $z(t_i)$ near $t_i=0:\ z(t_i)=z(0)+z'(0)t_i+\frac12z''(0)t_i^2+\frac16z'''(0)t_i^3+...,\ t_i=t,t_1,t_2,t_3$. Then (omitting the argument $(0)$ from our notation)
\[
\begin{aligned}
z(t_i)-z(t_j)&=z'(t_i-t_j)+\frac12z''(t_i^2-t_j^2)+\frac16z'''(t_i^3-t_j^3)+...\\
                 &=(t_i-t_j)(z'+\frac12z''(t_i+t_j)+\frac16z'''(t_i^2+t_it_j+t_j^2))+...
\end{aligned}
\]
and similarly
\[
\begin{aligned}
(z&(t_i)-z(t_j))^{-1}(z(t_k)-z(t_l))=\\
   &=\frac{t_k-t_l}{t_i-t_j}\Bigl(1+\frac12(z')^{-1}z''(t_i+t_j)+\frac16(z')^{-1}z'''(t_i^2+t_it_j+t_j^2)\Bigr)^{-1}\\
   &\qquad\qquad\qquad\Bigl(1+\frac12(z')^{-1}z''(t_k+t_l)+\frac16(z')^{-1}z'''(t_k^2+t_kt_l+t_l^2)\Bigr)+...\\
   &=\frac{t_k-t_l}{t_i-t_j}\Bigl(1-\frac12(z')^{-1}z''(t_i+t_j)-\frac16(z')^{-1}z'''(t_i^2+t_it_j+t_j^2)+\frac14((z')^{-1}z'')^2(t_i+t_j)^2\Bigr)\\
   &\qquad\qquad\qquad\Bigl(1+\frac12(z')^{-1}z''(t_k+t_l)+\frac16(z')^{-1}z'''(t_k^2+t_kt_l+t_l^2)\Bigr)+...\\
   &=\frac{t_k-t_l}{t_i-t_j}\Bigl(1+\frac12(z')^{-1}z''(t_k+t_l-t_i-t_j)+\frac16(z')^{-1}z'''(t_k^2+t_kt_l+t_l^2-t_i^2-t_it_j-t_j^2)\\
   &\qquad\qquad\qquad +\frac14((z')^{-1}z'')^2((t_i+t_j)^2-(t_k+t_l)(t_i+t_j))\Bigr)+...
\end{aligned}
\]
where we use $...$ to denote the elements of degree $3$ and higher in $t_i$. In particular, taking $t_i=t_2,\ t_j=t_1,\ t_k=t_1,\ t_l=t$, we obtain
\[
\begin{aligned}
(z&(t_2)-z(t_1))^{-1}(z(t_1)-z(t))=\\
   &=\frac{t_1-t}{t_2-t_1}\Bigl(1+(t-t_2)\bigl(\frac12(z')^{-1}z''+\frac16(z')^{-1}z'''(t+t_1+t_2)-\frac14((z')^{-1}z'')^2(t_2+t_1)\bigr)\Bigr)+...
\end{aligned}
\]
Similarly, with $t_i=t,\ t_j=t_3,\ t_k=t_3,\ t_l=t_2$, we have:
\[
\begin{aligned}
(z&(t)-z(t_3))^{-1}(z(t_3)-z(t_2))=\\
   &=\frac{t_3-t_2}{t-t_3}\Bigl(1+(t_2-t)\bigl(\frac12(z')^{-1}z''+\frac16(z')^{-1}z'''(t+t_2+t_3)-\frac14((z')^{-1}z'')^2(t+t_3)\bigr)\Bigr)+...
\end{aligned}
\]
Finally, taking the product of these two expressions we obtain
\[
f(t,t_1,t_2,t_3)=\frac{(t_1-t)(t_3-t_2)}{(t_2-t_1)(t-t_3)}\left(1+(t_2-t)(t_3-t_1)\left(\frac16(z'(0))^{-1}z'''(0)-\frac14((z(0)')^{-1}z(0)'')^2\right)\right)
\]
\end{proof}
Compare this formula with the formula (4.7) from the paper \cite{AZ}. 

We call the expression $Sch(z)=(z')^{-1}z'''-\frac32((z')^{-1}z'')^2$ \textit{the noncommutative Schwarzian of $z(t)$}. Just like the classical Schwarz derivative, this operator is invariant (up to conjugations) with respect to the M\"obius transformations in $\mathcal R^2$: this is the direct consequence of the method we derived this formula from the (operator) cross-ratio.

\subsection{Infinitesimal Ceva ratio}
The following expression is intended as a 2-dimensional analog of the Schwarzian operator. More accurately, Schwarz derivative can be regarded as the infinitesimal transformation of the cross-ratio under a diffeomorphism of the projective line. It is natural to assume that the role of cross-ratio in projective plane should in some sense be played by the Ceva theorem (see figure 1, part (\textit{b})). Thus here we try to find the infinitesimal part of the transformation of the Ceva ratio under a diffeomorphism; in a general case this is quite a difficult question, so we do it under certain additional conditions.

Let $\xi,\,\eta$ be two commuting vector fields on a manifold $M$, and let $f:M\to M$ be a self-map of $M$ such that $df(\xi)=\kappa\cdot\xi,\,df(\eta)=\kappa\cdot\eta$ for some smooth function $\kappa\in C^\infty(M)$. It follows from this condition, that $f$ maps integral trajectories of both fields and of the fields, equal to their linear combinations with constant coefficients. One can imagine this map as a ``change of coordinates along the 2-dimensional net'', or a generalized conformal map. However, we do not assume that these fields are linearly independent, they can even be proportional to each other.

 Let us consider the following expression: take any point $x$; let $\phi(t)$ and $\psi(s)$ be the one-parameter diffeomorphism families, generated by $\xi$ and $\eta$ respectively. Since these fields commute, the composition $\phi(-r)\circ\psi(r)=\psi(r)\circ\phi(-r)=:\theta(r)$ is the one-parameter family, corresponding to their difference $\zeta=\eta-\xi$. Consider now the infinitesimal ``triangle'' at $x$: first we move from $x$ to $\phi(\epsilon)(x)$, then from this point to $\psi(2\epsilon)(x)$; then we apply to this point $\theta(\epsilon)$ and $\theta(2\epsilon)$; and finally we apply twice the diffeomorphism $\psi(-\epsilon)$. By definition, we come to the point $x$ again, having spun a ``curvilinear triangle'' $ABC$ ($A=x,\,B=\phi(2\epsilon)(x),\ C=\psi(2\epsilon)(x)$) with points $K,L,M$ on its sides ($K=\phi(\epsilon)(x),\,L=(\phi(\epsilon)\circ\psi(\epsilon))(x),\,M=\psi(\epsilon)(x)$). If we use the inherent ``time'' along the trajectories of the vector fields to measure length along these trajectories, then the points $K,L$ and $M$ will be midpoints of the sides of $ABC$ and the standard Ceva relation will be trivially $1$:
 \[
 c(A,B,C;K,L,M)=\frac{AK}{KB}\cdot\frac{BL}{LC}\cdot\frac{CM}{MA}=1.
 \]
 Consider now the image of triangle $ABC$ under $f$: the points $K,L$ and $M$ will again fall on the ``sides'' of this image, however the lengths will be somehow distorted (in fact even the fields $df(\xi)=\kappa\cdot\xi$ and $df(\eta)=\kappa\cdot\eta$ need not be commuting). Let us now explore this ``distortion'' up to the degree $2$ in $\epsilon$: 
 \begin{proposition}
 Up to degree $2$ the difference between the distorted Ceva relation and $1$ is trivial; we put
 \[
 c(f(A),f(B),f(C);f(K),f(L),f(M))-1=:\epsilon^2S_3(f,\xi,\eta;x)+o(\epsilon^2),
 \]
 then
 \[
 S_3(f,\xi,\eta;x)=\frac56\frac{\kappa''_{\eta\eta}(x)-\kappa''_{\xi\xi}(x)}{\kappa(x)},
 \]
 where we use the standard notation $\kappa'_\xi=\xi(\kappa),\,\kappa'_\eta=\eta(\kappa)$.
 \end{proposition}
 \begin{proof}
 We compute:
 \[
 \begin{aligned}
 f(A)f(K)&=\epsilon\kappa(x)+\frac12\epsilon^2\kappa'_\xi(x)+\frac16\epsilon^3\kappa''_{\xi\xi}(x)+o(\epsilon^3),\\
 f(K)f(B)&=\epsilon\kappa(x+\epsilon\xi)+\frac12\epsilon^2\kappa'_\xi(x+\epsilon\xi)+\frac16\epsilon^3\kappa''_{\xi\xi}(x+\epsilon\xi)+o(\epsilon^3)\\
            &=\epsilon\kappa(x)+\frac32\epsilon^2\kappa'_\xi(x)+\frac53\epsilon^3\kappa''_{\xi\xi}(x)+o(\epsilon^3),\\
 f(M)f(A)&=-\epsilon\kappa(x)-\frac12\epsilon^2\kappa'_\eta(x)-\frac16\epsilon^3\kappa''_{\eta\eta}(x)+o(\epsilon^3),\\
 f(C)f(M)&=-\epsilon\kappa(x)-\frac32\epsilon^2\kappa'_\eta(x)-\frac53\epsilon^3\kappa''_{\eta\eta}(x)+o(\epsilon^3),\\
 f(B)f(L)&=\epsilon\kappa(x+2\epsilon\xi)+\frac12\epsilon^2\kappa'_\zeta(x+2\epsilon\xi)+\frac16\epsilon^3\kappa''_{\zeta\zeta}(x+2\epsilon\xi)+o(\epsilon^3)\\
            &=\epsilon\kappa(x)+2\epsilon^2\kappa'_\xi(x)+2\epsilon^3\kappa''_{\xi\xi}(x)\\
            &\quad+\frac12\epsilon^2\kappa'_\zeta(x)+\epsilon^3\kappa''_{\xi\zeta}(x)+\frac16\epsilon^3\kappa''_{\zeta\zeta}(x)+o(\epsilon^2)\\
            &=\epsilon\kappa(x)+\frac32\epsilon^2\kappa'_\xi(x)+\frac12\epsilon^2\kappa'_\eta(x)\\
            &\quad+\frac56\epsilon^3\kappa''_{\xi\xi}(x)+\frac16\epsilon^3\kappa''_{\eta\eta}(x)+\frac23\epsilon^3\kappa''_{\xi\eta}(x)+o(\epsilon^3)\\
f(L)f(C)&=\epsilon\kappa(x+\epsilon(\xi+\eta))+\frac12\epsilon^2\kappa'_\zeta(x+\epsilon(\xi+\eta))\\
            &\qquad\qquad\qquad+\frac16\epsilon^3\kappa''_{\zeta\zeta}(x+\epsilon(\xi+\eta))+o(\epsilon^3)\\
           &=\epsilon\kappa(x)+\frac32\epsilon^2\kappa'_\eta(x)+\frac12\epsilon^2\kappa'_\xi(x)\\
           &\quad+\frac12\epsilon^3\kappa''_{\xi\xi}(x)+\frac12\epsilon^3\kappa''_{\eta\eta}(x)+\epsilon^3\kappa''_{\xi\eta}(x)\\
           &\quad+\frac12\epsilon^3\kappa''_{\eta\eta}(x)-\frac12\epsilon^3\kappa''_{\xi\xi}(x)\\
           &\quad+\frac16\epsilon^3\kappa''_{\xi\xi}(x)+\frac16\epsilon^3\kappa''_{\eta\eta}(x)-\frac13\epsilon^3\kappa''_{\xi\eta}(x)+o(\epsilon^3)\\
           &=\epsilon\kappa(x)+\frac32\epsilon^2\kappa'_\eta(x)+\frac12\epsilon^2\kappa'_\xi(x)\\
           &\quad+\frac16\epsilon^3\kappa''_{\xi\xi}(x)+\frac56\epsilon^3\kappa''_{\eta\eta}(x)+\frac23\epsilon^3\kappa''_{\xi\eta}(x)+o(\epsilon^3).
 \end{aligned}
 \]
 Plugging these expressions into the formula for $c(A,B,C;K,L,M)$, we obtain the expression we need.
 \end{proof}
The analogy between this expression and the Schwarz derivative is quite evident. One can ask, if it is possible to extend it in any reasonable way to a more general situation when there are less restrictions on the diffeomorphism, and also if there exist a non-commutative version of this operator. We are going to address these questions in forthcoming papers.  
  

\section {Non-commutative Schwarzian and differential relations}
In this section we present an alternative construction of Schwarz derivative: we obtain it as an invariant of a system of ``differential equations'' on an algebra. In commutative case the relation of Schwarzian and differential equations is well-known, see for example \cite{OT2}. It is remarkable, that this construction also can be phrased in purely algebraic terms. We shall also discuss below some properties of this construction.

Consider the following system of linear ``differential equations'':
\begin{equation}
\label{eq:sys1}
\begin{cases}
f_1''+af_1'+bf_1&\!\!\!\!=0\\
f_2''+af_2'+bf_2&\!\!\!\!=0.
\end{cases}
\end{equation}
Here $a,\,b,\,f_1,\,f_2$ are elements of a division ring  $\mathcal R$, and ${}'$ denotes a linear differentiation in this ring, i.e. a linear endomorphism of $\mathcal R$ verifying the noncommutative Leibniz identity (a model example is the algebra of smooth operator-valued functions of one (real) variable, however one can plug in arbitrary algebra with a differentiation of any sort on it). 

Below we shall assume that all the elements we deal with are invertible if necessary. Using this assumption it is not difficult to solve the equations \eqref{eq:sys1} as a linear system on $a$ and $b$: multiplying the equations by $f_1^{-1}$ and $f_2^{-1}$ respectively and subtracting the second one from the first one we obtain (see [3])
\[
a=-(f_1''f_1^{-1}-f_2''f_2^{-1})(f_1'f_1^{-1}-f_2'f_2^{-1})^{-1}
\]
and similarly
\[
b=-(f_1''(f_1')^{-1}-f_2''(f_2')^{-1})(f_1(f_1')^{-1}-f_2(f_2')^{-1})^{-1}.
\]
We can rewrite these formulas a little:
\[
\begin{aligned}
a&=-(f_1''-f_2''f_2^{-1}f_1)(f_1'-f_2'f_2^{-1}f_1)^{-1}\\
b&=-(f_1''-f_2''(f_2')^{-1}f_1')(f_1-f_2(f_2')^{-1}f_1')^{-1}
\end{aligned}
\]
so that now it is evident that $a$ and $b$ can be expressed as $a=-q^1_{32},\ b=-q^2_{31}$, where $q^i_{jk}$ are right quasi-Pl\"ucker coordinates of the $3\times 2$-matrix $\begin{pmatrix}f_1 & f_1' & f_1''\\ f_2 & f_2' & f_2''\end{pmatrix}^T$. See section 1 
for details.

Observe that in the process of solving \eqref{eq:sys1} we obtained the expression:
\begin{equation}
\label{eq:inter2}
-b=af_1'f_1^{-1}+f_1''f_1^{-1}=af_2'f_2^{-1}+f_2''f_2^{-1}.
\end{equation}
(The expression is a special case for the formula from Proposition 4.8.1 from [2] rewritten for right quasi-Pl\"ucker coordinates. The proposition connects quasi-Pl\"ucker coordinates for matrices of different sizes.)
 
Thus
\[
af_1'f_1^{-1}f_2+f_1''f_1^{-1}f_2=af_2'+f_2''.
\]
Hence
\begin{equation}
\label{eq:inter1}
af_1(f_1^{-1}f_1'f_1^{-1}f_2-f_1^{-1}f_2')=f_1(f_1^{-1}f_2''-f_1^{-1}f_1''f_1^{-1}f_2).
\end{equation}
Now one have the formulas:
\[
(f^{-1})''=-(f^{-1}f'f^{-1})'=2f^{-1}f'f^{-1}f'f^{-1}-f^{-1}f''f^{-1},
\]
and
\[
(fg)''=f''g+2f'g'+fg'',
\]
for all $f,g\in A$; so
\[
(f^{-1}g)''=2f^{-1}f'f^{-1}f'f^{-1}g-2f^{-1}f'f^{-1}g'-f^{-1}f''f^{-1}g+f^{-1}g''.
\]
Thus on the right hand side of \eqref{eq:inter1} we have
\[
\begin{aligned}
f_1(f_1^{-1}f_2''-f_1^{-1}f_1''f_1^{-1}f_2)&=f_1(2f_1^{-1}f_1'f_1^{-1}f_1'f_1^{-1}f_2-2f_1^{-1}f_1'f_1^{-1}f_2'-f_1^{-1}f_1''f_1^{-1}f_2+f_1^{-1}f_2'')\\
                                                               &\quad-2f_1(f_1^{-1}f_1'f_1^{-1}f_1'f_1^{-1}f_2-2f_1^{-1}f_1'f_1^{-1}f_2')\\
                                                               &=f_1(f_1^{-1}f_2)''-2f_1'(f_1^{-1}f_1'f_1^{-1}f_2-f_1^{-1}f_2')\\
                                                               &=f_1(f_1^{-1}f_2)''+2f_1'(f_1^{-1}f_2)'
\end{aligned}
\]
On the other hand, on the left hand side of \eqref{eq:inter1} we have $-af_1(f_1^{-1}f_2)'$, so denoting $\varphi=f_1^{-1}f_2$ we get:
\begin{equation}
\label{eq:eq2}
af_1=-2f_1'-f_1\varphi''(\varphi')^{-1},
\end{equation}
or equivalently
\begin{equation}
\label{eq:eq2'}
a=-2f_1'f_1^{-1}-f_1\varphi''(\varphi')^{-1}f_1^{-1}.
\end{equation}
Here's a simple corollary of the formula \eqref{eq:eq2'}:
\begin{proposition}
\label{prop:1}
When the elements $f_i\in A$ are replaced by $\tilde f_i=hf_i,\ i=1,2$ for some $h\in \mathcal R$, then $a$ in the system \eqref{eq:sys1} should be replaced by $\tilde a=-2h'h^{-1}+hah^{-1}$.
\end{proposition}
\begin{proof}
Observe that $\varphi$ is not affected by the coordinate change $f_i\leftrightarrow \tilde f_i,\ i=1,2$. Now direct calculation with formula \eqref{eq:eq2'} shows
\[
\begin{aligned}
\tilde a&=-2{\tilde f_1}'{\tilde f_1}^{-1}-\tilde f_1\varphi''(\varphi')^{-1}{\tilde f_1}^{-1}\\
           &=-2(h'f_1+hf_1')f_1^{-1}h^{-1}-h(f_1\varphi''(\varphi')^{-1}f_1^{-1})h^{-1}\\
           &=-2h'h^{-1}+h(-2f_1'f_1^{-1}-f_1\varphi''(\varphi')^{-1}f_1^{-1})h^{-1}.
\end{aligned}
\]
\end{proof}
\begin{remark}\rm
It is worth to observe a striking similarity of the expression in proposition \ref{prop:1} and the gauge transformation of a linear connection (the unnecessary $2$ in front of $h'h^{-1}$ can be eliminated by considering $\alpha=\frac12a$).
\end{remark}
It is now our purpose to find the way $b$ changes, when $f_1,\,f_2$ are multiplied by $h$, at least under some additional assumptions on $h$. We begin with the simple observation:
\begin{corollary}
If $h$ verifies the ``differential equation'' $h'=\frac12ha$ then $\tilde a=0$.
\end{corollary}
Further, there's another simple consequence of the formula \eqref{eq:eq2}:
\begin{proposition}
Assume that $h$ verifies the equation $h'=\frac12ha$; denote $\tilde f_1=h f_1,\ \theta=\varphi''(\varphi')^{-1}$. Then
\[
{\tilde f_1}'=-\frac{\tilde f_1}{2}\theta.
\]
\end{proposition}
\begin{proof}
\[
{\tilde f_1}'=(h f_1)'=\frac12haf_1+hf_1'=\mbox{using equation \eqref{eq:eq2}}=-hf_1'-\frac12hf_1\theta+hf_1'=-\frac{\tilde f_1}{2}\theta.
\]
\end{proof}
Repeating the differentiation we see:
\begin{equation}
\label{eq:inter3}
{\tilde f_1}''=-\left (\frac{\tilde f_1}{2}\theta\right)'=\frac{\tilde f_1}{4}\theta-\frac{\tilde f_1}{2}\theta'.
\end{equation}
Finally, substituting these formulas in the first expression of \eqref{eq:inter2} we obtain the following result:
\begin{theorem}
If $h$ satisify the equation $h'=\frac12ha$ then the coordinate change $f_i\mapsto \tilde f_i=hf_i,\ i=1,2$ transforms the system \eqref{eq:sys1} in such a way that 
\[
\begin{aligned}
a&\mapsto 0\\
b&\mapsto \frac12\tilde f_1\left(\theta'-\frac12\theta\right){\tilde f_1}^{-1}
\end{aligned}
\]
where $\theta=\varphi''(\varphi')^{-1},\ \varphi=f_1^{-1}f_2$ and the equation $2{\tilde f_1}'+\tilde f_1\theta=0$ holds.
\end{theorem}
\begin{proof}
Since $a\mapsto 0$, we obtain from \eqref{eq:inter2}:
\[
b=-{\tilde f_1}''{\tilde f_1}^{-1}=\mbox{using \eqref{eq:inter3}}=-\left(\frac{\tilde f_1}{4}\theta-\frac{\tilde f_1}{2}\theta'\right){\tilde f_1}^{-1}
\]
\end{proof}
\begin{remark}\rm
Observe that in the commutative case the expression $\theta'-\frac12\theta$ coincides with the classical Schwarz differential of $\varphi$.
\end{remark}


\subsection{Generalized NC Schwarzian}
Let  $f$ and $g$ be two (invertible) elements of a division ring $\mathcal R$, equipped with a derivation $'$ (see previous section). We suppose that they satisfy so-called 
{\it left coefficients equations} $f''=F_1f$, $g''=F_2g$ for some $F_1,F_2\in\mathcal R$. 
We set $h:=fg^{-1}$ and $G:=F_1h-hF_2$.

\begin{theorem}\label{ncSch} 
If $G=0$ then we have the following relation:
\begin{equation}\label{NCSch}
h'''=(3/2)h''(h')^{-1}h'' - 2h'F_2
\end{equation}
(a {\it non-commutative analogue of the Schwarzian equation}.)
\end{theorem}

\begin{proof}
$$
h'=f'g^{-1}-hg'g^{-1},
$$
$$
h''=G-2h'g'g^{-1},
$$
$$
h'''=G'-2h''g'g^{-1}-2h'F_2+2h'(g'g^{-1})^2\ .
$$

One can express $g'g^{-1}=(1/2)(h')^{-1}(G-h'')$ and get
$$
h'''=(3/2)h''(h')^{-1}h'' - 2h'F_2-(3/2)h''(h')^{-1}G +(1/2)G(h')^{-1}(G-h'')\ .
$$

Let $f^{-1}f''= g^{-1} g''$ , i.e. $f, g$ are solutions of the same differential equation
with {\it  right coefficients}. Let $g'' = F g,$ i.e. $g$ is also a solution of a differential
equation with a left coefficient. Let $h = f g^{-1}$ . 
Then
$$
h'''-(3/2)h'' (h')^{-1} h'' = -2h'F .
$$
Note that that the left-hand side is stable under M\"{o}bius transform
$$
h \to  (ah + b)(ch + d)^{-1}
$$
where $a'=b'= c'= d'=0.$
\end{proof}

\begin{remark}
Consider the commutative analogue of (\ref{NCSch})
\begin{equation}\label{Sch}
(h')^{-1}h'''=(3/2)(h')^{-2}h''^2- 2F_2.
\end{equation}
This equality can be regarded as yet another definition of the Schwarzian ${\rm Sch}(h)$
$$
{\rm Sch}(h):= (h')^{-1}h'''-(3/2)(h')^{-2}h''^2 = - 2F_2.
$$
Hence, we obtain one more justification for calling {\it a NC Schwarzian of $h$} the following expression
\begin{equation}\label{NCSch1}
{\rm NCSch}(h):=(h')^{-1}h''' - (3/2)(h')^{-1}h''(h')^{-1}h'' 
\end{equation}
\end{remark}


\begin{remark}
In commutative case there exist the following famous version of KdV equation
\begin{equation}\label{SchKdV}
h_t = (h'){\rm Sch}(h)
\end{equation}
It is invariant under  the projective action of $SL_2$ and, when written as an evolution on the invariant $\rm{Sch}(h)$ it
becomes the "usual" KdV
\begin{equation}\label{KdV}
{\rm Sch}(h)_t = {\rm Sch}(h)^{'''} + 3{\rm Sch}(h)' {\rm Sch}(h).
\end{equation}
\end{remark}
Introducing two commuting derivatives  $\partial_x = '$ and $\partial_t$ of our skew-field $\mathcal R$ with respect to two distinguished elements $x$ and $t$ one can write the analogs of (\ref{SchKdV}):
\begin{equation}\label{NSchKdV}
h_t = (h'){\rm NSch}(h) = h''' - (3/2)h''(h')^{-1}h'' 
\end{equation}


\begin{remark}
The equation (\ref{NSchKdV}) has an interesting  geometric interpretation (specialisation) as the {\it Spinor Schwarzian-KdV equation} (see the equation (4.6) in \cite{MB}).





\end{remark}

\section{Some applications of NC cross-ratios.}
Let us briefly describe few possible applications of noncommutative cross-ratios, inspired by the classical constructions.

\subsection{Noncommutative leapfrog map}

Let $\mathbb P^1$ be the projective line over a noncommutative division ring $\mathcal R.$
Consider points five points $S_{i-1} , S_i , S_{i+1} , S_i^{-}$ and $S_i^{+}$ on $\mathbb P^1$. 
The theory of noncommutative cross-ratios (see theorem 3.4) implies that there exists a projective transformation sending
$$
\left(S_{i-1}, S_i, S_{i+1}, S_i^{-}\right) \to \left(S_{i+1} , S_i , S_{i-1} , S_i^{+}\right)
$$
(in this order!) if and only if the corresponding cross-ratios coincide:
$$
\begin{aligned}
\bigl(S_{i+1}&-S_i \bigr)^{-1}\left(S_i^{-} - S_i\right)\left(S_i^{-} - S_{i-1}\right)^{-1}\left(S_{i+1} - S_{i-1}\right)=\\
&= \lambda^{-1}\left(S_{i-1} - S_i\right)^{-1}\left(S_i^{+}-S_i\right)\left(S_i^{+} - S_{i+1}\right)^{-1}\left(S_{i-1} - S_{i+1}\right)\lambda
\end{aligned}
$$
where $\lambda \in R.$

Note that the factor $\left(S_{i+1}-S_{i-1}\right)$ appears in both sides of the equation
but with the different signs. It shows that in the commutative case one gets
the identity (5.14) from \cite{GSTV}; this map is integrable and constitutes a part of the pentagramm family of maps, see the next paragraph.
\begin{problem} 
It is very intriguing if the same properties exist in noncommutative case.
\end{problem}

\subsection{Noncommutative cross-ratios and the pentagramma mirificum}

\subsubsection{Classical 5-recurrence} 

There is a wonderful observation (known as the {\it Gauss Pentagramma mirificum}) that when a pentagramma is drawn on a unit sphere in $\mathbb R^3$ (see figure 2, where we do not observe the orthogonality of great circles) with successively orthogonal great circles with the lengths of inner side arcs $\alpha_i,\quad i=1,\ldots,5$ and one takes $y_i := \tan^2 (\alpha_i),$ then the following recurrence relation satisfies:
\begin{figure}[tb]
\includegraphics[scale =.65]{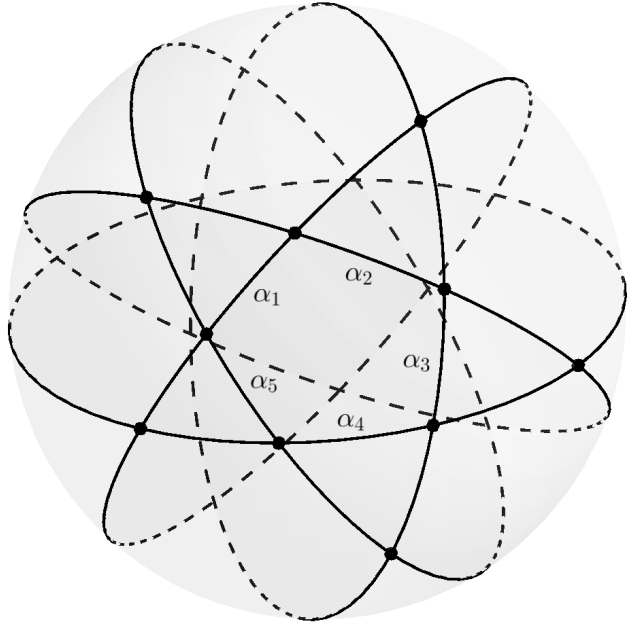}
\caption{Gauss' pentagramma mirificum}
\end{figure}

\begin{equation}\label{pentmirGauss}
y_i y_{i+1} = 1 + y_{i+3}, \quad {\rm mod} \quad \mathbb Z_5.
\end{equation}
Gauss has observed that the first three equations for  $i=1,2,3$ in \ref{pentmirGauss} completely define the last two equations for $i=4,5$.

It was discussed in  \cite{M-G} (which is our main source of the classical data for the Gauss Pentagramma Mirificum) that the variables $y_i$ can be expressed via the
classical cross-ratios:
$$y_i = [p_{i+1}, p_{i+2}, p_{i+3},p_{i+4}] = \frac{(p_{i+4}-p_{i+1})(p_{i+3}-p_{i+2})}{(p_{i+4}-p_{i+3})(p_{i+2}-p_{i+1})},$$
where $p_i = p_{i+5}$ are five points on real or complex projective line.
\begin{proposition}
Suppose that two consecutive points $y_i$ and $y_{i+1}$ (cyclically) are differents. Then the five cross-ratios $y_i$ satisfy the relation \eqref{pentmirGauss}.
\end{proposition}

It was remarked in  \cite{M-G} that after renaming $x_1=y_1,\,x_2=y_4,\,x_3=y_2,\,x_4=y_5,\,x_5=y_3,$  the variables $x_i,\ i=1,\ldots,5$ satisfy the famous \textit{pentagon recurrence}:
$$x_{i-1}x_{i+1}= 1+x_{i}.$$
It is also known, that this construction is closely related to cluster algebras, see \cite{FR} for further details.

\subsubsection{Non-commutative analogues} 

Let $\mathcal R$ be an associative division ring. In [6] (see sections 2 and 3) we defined the cross-ratio $\kappa(i, j, k, l)$ four vectors $i, j, k, l \in \mathcal R^2.$ Recall that
$$
\kappa(i, j, k, l)= q^j_{kl}q^i_{lk}
$$
where $q^k_{ij}$
is the corresponding quasi-Pl\"{u}cker coordinate. In particular, $q^k_{ij}$ 
and $q^k_{ji}$ are inverse to each other and
$$
\kappa(j, i, l, k) = q^i_{lk}\kappa(i, j, k,l)q^i_{kl},\ \kappa(k,l, i, j) = q^j_{ik}\kappa(i, j, k, l)q^j_{ki}.
$$
We set $\overline{\kappa(i, j, k,l)} = \kappa(j, i, k, l)$.

Let now $i, j, k, l, m$ be five vectors in $\mathcal R^2.$ We start with multiplicative relations for their cross-ratios. All these relations are redundant in the commutative case. 
$$
\begin{aligned}
\kappa(i, j, k, l)q^{i}_{km} \kappa(i, k, m, l) q^{i}_{mk}& = q^{j}_{kl}\overline{\kappa(i,k,m,l)}\, \overline{\kappa(i, j,k, l)}q^{j}_{lk},\\
q^{l}_{ mk }\kappa(i, j, k, l)q^{l}_{ ki} \kappa(l, k, i, m)q^{l} _{im}& = \overline{\kappa(l, k, i, m)}q^{k}_{ ml} \overline{\kappa(i, j, k, l)}q^{k}_{lm }.
\end{aligned}
$$

Noncommutative versions of the {\it pentagramma mirificum} relations can be written as follows: 
$$\kappa(i, j, k, l)q^{i}_{kj}\kappa(m, l, j, i)q^{i}_{jk}  =  1 - \kappa(m, j, k, i) , $$
$$\kappa(i, j, k, l)q^{l}_{ki} \kappa(l, k, i, m)q^{i}_{jk} = 1 - \kappa(l, j, k, m) , $$
$$q^{l}_{ jk} \kappa(i, j, k, l)q^{l}_{kj}\kappa(m, l, j, i) = 1 - \kappa(i, k, j, m) . $$

For five vectors $1, 2, 3, 4, 5$ in $\mathcal R^2$ we  set
$$x_1 = -\kappa(1, 2, 3, 4),\ x_2 = -\kappa(5, 2, 3, 1),\ x_3 = -\kappa(5, 4, 2, 1),
$$
$$x_4 = -\kappa(3, 4, 2, 5),\ x_5 = -\kappa(3, 1, 4, 5) .$$

Then 
$$x_1 q^1_{32}x_3 q ^1_{23} =  1 + x_2, \quad  x_4 q^5_{23}x_2 q ^5_{32} =  1 + x_3, $$
$$x_3 q^5_{24}x_5 q ^5_{42} =  1 + x_4, \quad  x_6 q^3_{42}x_4 q ^3_{24} =  1 + x_5, $$
$$x_5 q^3_{41}x_7 q ^3_{14} =  1 + x_6,$$
where $x_6 := \bar x_1$ and $x_7 := \bar x_2 .$
Note the different order for even and odd left hand sides. So , we have an {\it 5-antiperiodicity}, i.e. the periodicity
up to the anti-involution $x_{k+5} = \bar x_k$

Also, the relations with odd left hand parts imply the relations for even left hand parts as in the commutative case.

\begin{remark}
There is an important "continuous limit" of "higher pentagramma" maps on polygons in $\mathbb P^n$  which is the Boussinesq (or generalized $(2, n+ 1)-$KdV hierarchy) equation (\cite{KhS}).
\end{remark}

\begin{problem}
What is a non-commutative "higher analogue" of pentagramma recurrences? Is there a related  non-commutative integrable analogue of the Boussinesq equation?
\end{problem}

We hope to return to these questions in our future paper devoted to new examples of NC integrable systems (\cite{RRS}).

\bigskip\noindent
Vladimir Retakh\\
Department of Mathematics\\
Rutgers University\\
Piscataway, New Jersey 08854, USA\\
e-mail: vretakh@math.rutgers.edu

\medskip\noindent
Vladimir Roubtsov\\
Maths Department, University of Angers\\
Building I\\
Lavoisier Boulevard\\
Angers, 49045, CEDEX 01, France\\
e-mail: volodya@univ-angers.fr

\medskip\noindent
Georgy Sharygin\\
Department of Mathematics and Mechanics\\
Moscow State (Lomonosov) University\\
Leninskie Gory, d. 1\\
Moscow 119991, Russia\\
e-mail: sharygin@itep.ru

\end{document}